\documentclass[amscd,amssymb,11pt]{amsart}

\usepackage{amssymb}
\usepackage[all]{xy}

\newtheorem{thm}{Theorem}[section]
\newtheorem{lem}[thm]{Lemma}
\newtheorem{cor}[thm]{Corollary}
\newtheorem{prop}[thm]{Proposition}

\newtheorem{prop/def}[thm]{Proposition/Definition}

\theoremstyle{definition}

\newtheorem{defn}[thm]{Definition}
\newtheorem{defns}[thm]{Definitions}

\newtheorem{ex}[thm]{Example}

\newtheorem{addendum}[thm]{Addendum}

\theoremstyle{remark}

\newtheorem{rem}[thm]{Remark}

\numberwithin{equation}{section}

\newcommand{\thmref}[1]{Theorem~\ref{#1}}

\newcommand{\secref}[1]{\S\ref{#1}}

\newcommand{\propref}[1]{Proposition~\ref{#1}}
\newcommand{\lemref}[1]{Lemma~\ref{#1}}

\newcommand{\exref}[1]{Example~\ref{#1}}

\newcommand{\adref}[1]{Addendum~\ref{#1}}

\newcommand{\colim}{\operatorname*{colim}}

\newcommand{\Hom}{\operatorname{Hom}}

\newcommand{\coker}{\operatorname{coker}}

\newcommand{\A}{{\mathcal  A}}

\newcommand{\C}{{\mathcal  C}}
\newcommand{\Ho}{{\mathcal  H}}
\newcommand{\M}{{\mathcal  M}}

\newcommand{\F}{{\mathbb  F}}

\newcommand{\Z}{{\mathbb  Z}}

\newcommand{\Q}{{\mathbb  Q}}

\newcommand{\n}{\bold n}

\newcommand{\ra}{\rightarrow}
\newcommand{\xra}{\xrightarrow}
\newcommand{\la}{\leftarrow}

\newcommand{\hra}{\hookrightarrow}

\newcommand{\ora}{\twoheadrightarrow}

\begin{document}

\title[Hopf Algebras]{Split Hopf algebras, quasi-shuffle algebras, and the cohomology of $\Omega \Sigma X$}

\author[Kuhn]{Nicholas J.~Kuhn}
\address{Department of Mathematics \\ University of Virginia \\ Charlottesville, VA 22904}
\email{njk4x@virginia.edu}

\date{July 9, 2019.}


\begin{abstract}

Let $A$ and $B$ be two connected graded commutative $k$--algebras of finite type, where $k$ is a perfect field of positive characteristic $p$.  We prove that the quasi--shuffle algebras generated by $A$ and $B$ are isomorphic as Hopf algebras if and only if $A$ and $B$ are isomorphic as graded $k$--vector spaces equipped with a Frobenius ($p^{th}$--power) map.

For the hardest part of this analysis, we work with the dual construction, and are led to study connected graded cocommutative Hopf algebras $H$ with two additional properties: $H$ is free as an associative algebra, and the projection onto the indecomposables is split as a morphism of graded $k$--vector spaces equipped with a Verschiebung map.

Building on work on non--commutative Witt vectors by Goerss, Lannes, and Morel, we classify such free, `split' Hopf algebras.

A topological consequence is that, if $X$ is a based path connected space, then the Hopf algebra $H^*(\Omega \Sigma X;k)$ is determined by the stable homotopy type of $X$.

We also discuss the much easier analogous characteristic 0 results, and give a characterization of when our quasi--shuffle algebras are polynomial, generalizing the so--called Ditters conjecture.
\end{abstract}

\maketitle
\section{Introduction} \label{introduction}

\subsection{An overview}

Let $k$ be a perfect field of characteristic $p$.  We study the functor sending a connected graded commutative finite type $k$--algebra $A$ to the quasi--shuffle algebra $J^{\vee}(A)$, the cofree Hopf algebra cogenerated by $A$.  This is of interest in algebraic topology since if $A = H^*(X;k)$, then $J^{\vee}(A) = H^*(\Omega \Sigma X;k)$.

By forgetting structure, $A$ is an $F$--module: a non-negatively graded $k$--vector space equipped with a Frobenius ($p^{th}$--power) map.  In \thmref{char p theorem}, we prove that $J^{\vee}(A)$ is isomorphic to $J^{\vee}(B)$ as Hopf algebras if and only if $A$ is isomorphic to $B$ as $F$--modules.

The result just stated overlaps with results in \cite{newman radford}, but our route to proving this is very different than that taken by Newman and Radford, and yields new understanding about the structure of these Hopf algebras.   For the hardest part of our analysis, we work with the dual construction, $J(C)$, the free Hopf algebra generated by a cocommutative coalgebra $C$.  When $H = J(C)$, $H$ is an example of a graded connected Hopf algebra $H$ over $k$ with two additional properties:
\begin{itemize}
\item $H$ is free as an algebra.
\item $\bar H \ra QH$ is split in the category $V$--modules.
\end{itemize}
Here a $V$--module is a non-negatively graded $k$--vector space equipped with a Verschiebung map, $\bar H$ is the positive part of $H$ (and the kernel of the counit), and $QH = \bar H /\bar H^2$ is the module of indecomposables.

Building on work on non--commutative Witt vectors by Goerss, Lannes, and Morel \cite{glm92}, we classify such free, `split' Hopf algebras, and prove that they have a certain lifting property.

A topological consequence is that, if $X$ is a connected space, then the Hopf algebra $H^*(\Omega \Sigma X;k)$ is determined by the stable homotopy type of $X$.

Other results discussed in the paper, dependent only on classic Hopf algebra theory, include the characteristic 0 analogues of the above results, and a characterization of when quasi--shuffle algebras are polynomial, generalizing the so--called Ditters conjecture about the ring of quasisymmetric functions.

In the rest of this introduction, we describe our results with more precision and detail.

\subsection{Some categories and two functors}

Let $k$ be a field, and let $\M(k)$ be category of non-negatively graded $k$--vector spaces of finite type: $M = \{M_n\ | \ n = 0,1,2,\dots\}$ with $M_n$ finite dimensional for all $n$.  We let $\bar M$ be the positive part of $M$, and we say that $M$ is {\em reduced} if $M_0 = 0$.

We will be studying objects in the following categories.
\begin{itemize}
\item $\A^*(k)$: connected, commutative algebras in $\M(k)$.
\item $\C_*(k)$: connected, cocommutative coalgebras in $\M(k)$.
\item $\Ho^*(k)$: connected, commutative  Hopf algebras in $\M(k)$.
\item $\Ho_*(k)$: connected, cocommutative Hopf algebras in $\M(k)$.
\end{itemize}

\begin{defn} Let $J: \C_*(k) \ra \Ho_*(k)$ be left adjoint to the forgetful functor.
\end{defn}
Explicitly, $J(C)$ will be the tensor algebra $T(C) = \bigoplus_{k=0}^{\infty} \bar C^{\otimes k}$ as an algebra, with coproduct induced by the coproduct $\Delta: C \ra C \otimes C$.  See \secref{quasi shuffle sec} for more detail.

The dual definition goes as follows.

\begin{defn} Let $J^{\vee}: \A^*(k) \ra \Ho^*(k)$ be right adjoint to the forgetful functor.
\end{defn}
Explicitly, $J^{\vee}(A)$ will be the cotensor coalgebra $T(A) = \bigoplus_{k=0}^{\infty} \bar A^{\otimes k}$ as a coalgebra, with product induced by the product $\nabla: A \otimes A \ra A$.

The coalgebra structure on $J^{\vee}(A)$ is easily understood: it is just the deconcatenation of tensors.  The product on $J^{\vee}(A)$ is a bit more complicated.  It seems to have been  first studied by K.~Newman and D.~Radford in the 1970's \cite{newman radford}, and then  rediscovered decades later by various authors, including J.-L.~Loday \cite{loday}, who called it the {\em quasi--shuffle product}.  We give our own quick exposition of this product in \secref{quasi shuffle sec}.

\begin{ex} \label{ex 1} Let $A = k[t]$, a polynomial algebra with generator $t$ in grading 2.  Then $J^{\vee}(A) = QSym(k)$, the much studied Hopf algebra of {\em quasi--symmetric functions} (over the field $k$). (See \cite{mason} for a recent survey.)

Dually, let $C$ have basis given by elements $t_k$, $k \geq 0$, of degree $2k$ with coproduct $\Delta(t_k) = \sum_{i+j=k} t_i \otimes t_j$.  Then $J(C)=NSym(k)$, the Hopf algebra of {\em non-symmetric functions} (over $k$). As an algebra, $J(C) = T\langle t_1,t_2, \dots\rangle$, the free associative algebra generated by $t_k$ for $k\geq 1$ ($t_0 = 1$).

The classic ring of symmetric functions $Sym(k)= k[t_1,t_2,\dots]$ is a self dual Hopf algebra, and there is a projection of Hopf algebras
$ NSym(k) \twoheadrightarrow Sym(k)$,
and a dual embedding of Hopf algebras
$ Sym(k) \hookrightarrow QSym(k)$.
\end{ex}

\subsection{The functors topologically realized}

If $X$ is a based topological space, let $\Sigma X$ be its reduced suspension, and $\Omega X$ be the space of based loops in $X$.  The functor $\Sigma$ is left adjoint to $\Omega$, and thus $\Omega \Sigma X$ admits an interpretation as the free based loopspace generated by $X$.

Classic results of Ioan James from the 1950's \cite{james} imply that if $X$ is a connected CW complex of finite type then the natural map $X \ra \Omega \Sigma X$ induces an isomorphism of Hopf algebras
$$ J(H_*(X;k)) \xra{\sim} H_*(\Omega \Sigma X;k).$$
Dually, one obtains natural isomorphisms of Hopf algebras
$$  H^*(\Omega \Sigma X;k) \xra{\sim} J^{\vee}(H^*(X;k)).$$

\begin{ex}  As observed by Andrew Baker and Birgit Richter in \cite{baker richter}, \exref{ex 1} can be topologically realized.   There are Hopf algebra isomorphisms $Sym(k) \simeq H^*(BU;k) \simeq H_*(BU;k)$, $ NSym(k) \simeq H_*(\Omega \Sigma \mathbb C \mathbb P^{\infty}; k)$,
and $ QSym(k) \simeq H^*(\Omega \Sigma \mathbb C \mathbb P^{\infty}; k)$. (Here $ U = \bigcup_n U(n)$ is the infinite unitary group, and $BU$ is its classifying space.)

Since $BU$ has the structure of a loopspace, the map $\mathbb C \mathbb P^{\infty} = BU(1) \ra BU$ extends to a map of loopspaces
$ \Omega \Sigma \mathbb C \mathbb P^{\infty} \ra BU$.
This map then induces the Hopf algebra projection $ NSym(k) \twoheadrightarrow Sym(k)$ in homology, and the Hopf algebra embedding $ Sym(k) \hookrightarrow QSym(k)$ in cohomology.
\end{ex}

\subsection{When are two quasi-shuffle algebras isomorphic?}

Given two commutative algebras $A,B \in \A^*(k)$, one might ask when $J^{\vee}(A)$ is isomorphic to $J^{\vee}(B)$.

This question comes in various versions, as one could consider $J^{\vee}(A)$ as a Hopf algebra, or as an algebra, or as a coalgebra.

From the construction, the one evident statement here is that $J^{\vee}(A)$ is isomorphic to $J^{\vee}(B)$ as coalgebras if and only if $A$ is isomorphic to $B$ as graded $k$--vector spaces, as, in that case, $J^{\vee}(A)$ and $J^{\vee}(B)$ will be cofree coalgebras of the same size.

Classical work on Hopf algebras by Milnor and Moore (and others earlier) \cite{milnor moore} implies a definitive answer to all versions of the question, when $k$ has characteristic 0.

\begin{thm} \label{char 0 theorem}  Let $k$ be a field of characteristic 0.  Given $A, B \in \A^*(k)$, the following are equivalent. \\

\noindent{\bf (a)} $\bar A \simeq \bar B$, as graded $k$--vector spaces. \\

\noindent{\bf (b)} $J^{\vee}(A) \simeq J^{\vee}(B)$, as algebras. \\

\noindent{\bf (c)} $J^{\vee}(A) \simeq J^{\vee}(B)$, as Hopf algebras.
\end{thm}

To explain the characteristic $p$ version of this, we need a definition.

\begin{defn} \label{F module defn} Let $k$ be a perfect field of characteristic $p$.  An $F$--module is a reduced graded $k$--module $M \in \M(k)$ equipped with Frobenius maps:
$ F: M_n \ra M_{pn}$, for all $n$ if $p=2$ and for even $n$ if $p$ is odd, such that $F(x+y) = F(x) + F(y)$ and $F(\lambda x) = \lambda^pF(x)$ for all $\lambda \in k$.  A morphism between two $F$--modules is a morphism in $\M(k)$ commuting with the Frobenius maps.  The $F$--modules then form an abelian category  which we denote by $F$--$\M(k)$.
\end{defn}

\begin{ex} By forgetting structure, the augmentation ideal $\bar A$ of a commutative algebra $A \in \A^*(k)$ can be viewed as an $F$--module, with $ F: \bar A_n \ra \bar A_{pn}$ defined by $F(x) = x^p$.
\end{ex}

\begin{thm} \label{char p theorem}  Let $k$ be a perfect field of characteristic $p>0$.  Given $A, B \in \A^*(k)$, the following are equivalent. \\

\noindent{\bf (a)} $\bar A \simeq \bar B$, as $F$--modules. \\

\noindent{\bf (b)} $J^{\vee}(A) \simeq J^{\vee}(B)$, as algebras. \\

\noindent{\bf (c)} $J^{\vee}(A) \simeq J^{\vee}(B)$, as Hopf algebras.
\end{thm}

Though this statement is a nice characteristic $p$ analogue of \thmref{char 0 theorem}, the {\em proof} of the most interesting implication here is intrinsically {\em not} analogous.  Implication (c) $\Rightarrow$ (b) is clear. Classic results about graded Hopf algebras \cite{milnor moore} show that (b) is equivalent to the statement that $J^{\vee}(A) \simeq J^{\vee}(B)$ as $F$--modules, and a bit of bookkeeping shows that this is equivalent to (a).  But we will see that proving the last implication, (b) $\Rightarrow$ (c) (or (a) $\Rightarrow$ (c)), leads us to a new classification result.

Note that, when (b) holds, not only will $J^{\vee}(A) \simeq J^{\vee}(B)$ be isomorphic as algebras, but also as as coalgebras, as they will be cofree coalgebras of the same size.  As we soon relate, we have an example that reveals there is still work to be done to conclude that $J^{\vee}(A) \simeq J^{\vee}(B)$ will be isomorphic as Hopf algebras.

\begin{rem}  The 1979 paper \cite{newman radford} has results overlapping with \thmref{char p theorem}. In particular, \cite[Corollary 3.8(a)]{newman radford} is (roughly)\footnote{\cite[Corollary 3.8(a)]{newman radford} has as a hypothesis that $F^n(\bar A)=0$ for some $n$, but \cite[Proposition 3.7]{newman radford} would apply if one just had that all elements in $\bar A$ are nilpotent.} our implication (a) $\Leftrightarrow$ (c), under the side hypothesis that every element $x \in \bar A$ is nilpotent.
\end{rem}

\subsection{The dual formulation of \thmref{char p theorem}}

The theorems in the last subsection clearly are equivalent to dual versions.  We leave the formulation of the dual form of \thmref{char 0 theorem} to the reader, but it will be useful to explicitly discuss the dual version of \thmref{char p theorem}.

\begin{defn} \label{V module defn} Let $k$ be a perfect field of characteristic $p$.  A $V$--module is a reduced graded $k$--module $M \in \M(k)$ equipped with Verschiebung maps:
$ V: M_{pn} \ra M_{n}$, for all $n$ if $p=2$ and for even $n$ if $p$ is odd, such that $V(x+y) = V(x) + V(y)$ and $V(\lambda x) = \lambda^{1/p}V(x)$ for all $\lambda \in k$.  As before, $V$--modules form an abelian category, which we denote $V$--$\M(k)$.
\end{defn}

\begin{ex} By forgetting structure, the positive part, $\bar C$, of a cocommutative coalgebra $C \in \C_*(k)$ can be viewed as an $V$--module.
\end{ex}

\begin{thm} \label{char p theorem dualized}  Let $k$ be a perfect field of characteristic $p>0$.  Given $C, D \in \C_*(k)$, the following are equivalent. \\

\noindent{\bf (a)} $\bar C \simeq \bar D$, as $V$--modules. \\

\noindent{\bf (b)} $J(C) \simeq J(D)$, as coalgebras. \\

\noindent{\bf (c)} $J(C) \simeq J(D)$, as Hopf algebras.
\end{thm}

As before, implication $(c) \Rightarrow (b)$ is clear, and classic Hopf algebra theory, together with some bookkeeping, shows that $(b) \Leftrightarrow (a)$.  Note that when $(b)$ holds, $J(C)$ and $J(D)$ will be isomorphic free algebras that are also isomorphic as coalgebras.

\subsection{A cautionary example, topologically realized}

\begin{ex}  \label{bad example}  Working over the field $k=\F_2$,  we consider the tensor algebra $T(x,y,z)$, with $|x|=1$, $|y|=2$, and $|z|=4$, made into a Hopf algebra in two different ways: call these $H_1$ and $H_2$.  In both of these, $x$ and $y$ are primitive, i.e. $\Delta(x) = x \otimes 1 + 1 \otimes x$ and $\Delta(y) = y \otimes 1 + 1 \otimes y$.  In $H_1$, we let $\Delta(z) = z\otimes 1 + y \otimes y + 1 \otimes z$, and in $H_2$, we let $\Delta(z) = z\otimes 1 + x^2 \otimes x^2 + 1 \otimes z$.  Then $H_1$ and $H_2$ are isomorphic as both algebras and coalgebras, but not as Hopf algebras.

These Hopf algebras can be topologically realized:
$$H_1 \simeq H_*(\Omega \Sigma (S^1\vee \mathbb C\mathbb P^2);\F_2) \text{ \  and \  }
H_2 \simeq H_*(\Omega (S^3\vee C(\eta \circ \Sigma \eta));\F_2).$$
Here $\eta: S^3 \ra S^2$ is the Hopf map, and $C(\eta \circ \Sigma \eta)$ is the mapping cone of the composite $S^4 \xra{\Sigma \eta} S^3 \xra{\eta} S^2$.
\end{ex}

We will supply details in \S \ref{bad example sec}.

\begin{rem}  The Hopf algebras here are noncommutative versions of those appearing in \cite[Example 2.3]{klw}.  It was clear that $H_1$ could be topologically realized, as it is $J(C)$, where $C=H_*(S^1 \vee \mathbb C \mathbb P^2;\F_2)$; it was amusing to realize that $H_2$ could be also.
\end{rem}

\subsection{Indecomposables and split Hopf algebras}

To prove the implication (b) $\Rightarrow$ (c) (or (a) $\Rightarrow$ (c)) in \thmref{char p theorem dualized}, we need to use that the Hopf algebras we are considering -- Hopf algebras of the form $J(C)$ -- have some structure not present in Hopf algebras like $H_2$ in \exref{bad example}.

To describe this, we remind the reader of some standard terminology and notation.  If $H$ is a graded Hopf algebra, $PH \subset \bar H$ is the module of primitives, and $QH = \bar H/\bar H^2$ is the module of indecomposables.

\begin{lem}  Let $k$ be a perfect field of characteristic $p>0$.  If $H \in \Ho^*(k)$ then $PH \subset \bar H$ is a sub--$F$--module.  If $H \in \Ho_*(k)$ then $QH$ is a quotient of $\bar H$ as a $V$--module.
\end{lem}
\begin{proof}  The first statement is just the well known observation that the $p^{th}$ power of a primitive is again primitive.  The second statement is dual to this.
\end{proof}

Our needed extra structure on our Hopf algebras is described in the next definition.

\begin{defn}  We say that $H \in \Ho_*(k)$ is {\em split} if the quotient map $\bar H \ra QH$ has a section in $V$--$\M(k)$.
\end{defn}

\begin{lem} \label{JC is split lem}  Given $C \in \C_*(k)$, $J(C)$ is canonically split.
\end{lem}
\begin{proof}
 The natural inclusion $C \ra J(C)$ is a morphism of coalgebras, thus $\bar C \ra \overline{J(C)}$ is a morphism of $V$--modules, and we see that the composite $\bar C \ra \overline{J(C)} \ra QJ(C)$ is an isomorphism in $V$--$\M(k)$.
\end{proof}

\begin{ex}  With $H_2$ as in \exref{bad example}, $QH_2 = \langle \bar x,\bar y, \bar z\rangle$ with zero Verschiebung.  The natural lift of $\bar z$ -- the generator $z$ -- satisfies $V(z) = x^2$, and from this it is not hard to check that $\bar H_2 \ra QH_2$ does not admit a section as $V$--modules.  Thus $H_2$ is not split.
\end{ex}

\begin{ex}  Given $H \in \Ho_*(k)$, if $QH$ is a projective object in $V$--$\M(k)$, then $H$ is clearly split.
\end{ex}

We have the following rigidity/classification theorem for split free Hopf algebras.

\begin{thm} \label{split thm}  Let $k$ be a perfect field of characteristic $p>0$. \\

\noindent{\bf (a)} Given any $V$--module $M$, there is a unique $H(M) \in \Ho_*(k)$ that is split, free as an algebra, such that $QH(M) \simeq M$ as $V$--modules.  \\

\noindent{\bf (b)}  If $M \simeq \bigoplus_i M(i)$ is a decomposition of $M$ as a direct sum of indecomposable $V$--modules, then there is decomposition $H(M) \simeq \underset{i}{*} H(M(i))$ as the coproduct (free product) of indecomposable Hopf algebras.
\end{thm}

We say more about the indecomposable $H(M)$'s in \secref{split thm proof sec}. The uniqueness in this theorem is up to Hopf algebra isomorphism, and follows from the following lifting theorem.

\begin{thm} \label{lifting thm}  Let $M \in V\text{--}\M(k)$, and $H \in \Ho_*(k)$.  Given a $V$--module morphism $g: M \ra Q(H)$ that admits a lift to $\bar H$, there exists $G: H(M) \ra H$ in $\Ho_*(k)$ such that $QG = g: M \ra QH$.
\end{thm}

\thmref{split thm} lets us easily prove the hard implication in \thmref{char p theorem dualized}.

\begin{proof}[Proof that $(a) \Rightarrow (c)$ in \thmref{char p theorem dualized}] As observed in \lemref{JC is split lem}, $J(C)$ is split with $QJ(C) \simeq \bar C$ as $V$--modules.  Thus \thmref{split thm} implies that $J(C) \simeq H(\bar C)$ as Hopf algebras.  Thus if $\bar C \simeq \bar D$ as $V$--modules, there will be Hopf algebra isomorphisms $$J(C) \simeq H(\bar C) \simeq H(\bar D) \simeq J(D).$$
\end{proof}

\subsection{Organization of the rest of the paper}

Section \ref{background section} has some background material we will need on the categories $F$--$\M(k)$, $V$--$\M(k)$, and $\Ho_*(k)$.  Regarding $V$--$\M(k)$, when $k$ is a perfect field, $V$--modules are easily classified as sums of some basic examples $M(n,j)$, and some of these $V$--modules are projective.

In \secref{split thm proof sec}, we prove Theorems \ref{split thm} and \ref{lifting thm}.  Much of the heavy lifting was already done by Goerss, Lannes, and Morel in \cite{glm92}: for projective $V$--modules $M$, Hopf algebras $H(M) \in \Ho_*(k)$ with properties as in the two theorems are constructed. When $M$ is also indecomposable, the authors of \cite{glm92} interpret $H(M)$ as a Hopf algebra of noncommutative Witt vectors.  Using results from \secref{background section}, we are able to extend their results to general $V$--modules, thus proving \thmref{split thm}.   (In truth, the results in \cite{glm92} are only proved when $k$ is the prime field $\F_p$; we explain how to extend these to all perfect fields $k$.)

Section \ref{bad example sec} has the details of our cautionary example, \exref{bad example}.

We say a bit about quasishuffle algebras in \secref{quasi shuffle sec}.  This is followed, in \secref{easy proofs sec}, by the proofs of the `easy' implications in \thmref{char p theorem} and the proof of \thmref{char 0 theorem}: those following from the classical results of Milnor and Moore \cite{milnor moore}.

The last section, \secref{applications sec}, has some examples and related results.  This includes a new characterization of primitively generated Hopf algebras, and a characterization of when the quasi-shuffle algebra $J^{\vee}(A)$ is polynomial, generalizing the so--called Ditters conjecture: $QSym(k)$ is a polyonomial algebra. (A careful reader will see that our argument is essentially that of Baker and Richter in \cite{baker richter}.)  We also note a topological application: if two based spaces $X$ and $Y$ are stably homotopy equivalent, then $H_*(\Omega \Sigma X;k)$ and $H_*(\Omega \Sigma Y;k)$ are isomorphic Hopf algebras for all fields $k$.

\subsection{Thanks}  The starting point of this project was learning about the Ditters conjecture from Andy Baker back in 2014.  I had a chance to share some of my (rather naive) thoughts about this topic at the conference on Group Actions and Algebraic Combinatorics held at Herstmonceax, England in July, 2016.  My main theorems were proved during a visit to Sheffield University in spring, 2017.  John Palmieri also assisted me with references in the Hopf algebra literature.

\section{Background material} \label{background section}

\subsection{The category $\M(k)$}  Recall that $k$ is a field, and $\M(k)$ is the category of non-negatively graded $k$--vector spaces of finite type: $M = \{M_n\ | \ n = 0,1,2,\dots\}$ with $M_n$ finite dimensional for all $n$.  As in the introduction, we say that $M$ is {\em reduced} if $M_0 = 0$.

In the `usual way', $\M(k)$ is a symmetric $k$--linear tensor category: $$(M \otimes N)_k = \bigoplus_{i+j=k} M_i \otimes N_j$$ with braiding isomorphism $\tau: M \otimes N \xrightarrow{\sim} N \otimes M$ defined by $\tau(x \otimes y) = (-1)^{ij}y \otimes x$ for $x \in M_i$ and $y \in N_j$.

This structure allows one to define algebras, coalgebras, and Hopf algebras in $\M(k)$.  For example, an object in $\Ho_*(k)$, the category of connected, cocommutative Hopf algebras in $\M(k)$, consists of $H \in \M(k)$, with $H_0 = k$, equipped with unit $\eta: k \ra H$, counit $\epsilon: H \ra k$, multiplication $\nabla: H \otimes H \ra H$, and comultiplication $\Delta: H \ra H \otimes H$, satisfying appropriate properties.

Finally, one has a duality functor $(\text{\hspace{.1in}})^{\vee}: \M(k)^{op} \ra \M(k)$ given by taking levelwise dual vector spaces.  This is an equivalence of symmetric $k$--linear tensor categories, and induces equivalences $\C(k)_*^{op} \simeq \A^*(k)$ and $\Ho_*(k)^{op} \simeq \Ho^*(k)$.

\subsection{The categories $F$--$\M(k)$ and $V$--$\M(k)$}

Let $k$ be a field of characteristic $p$.

Recall that if $W$ is a $k$--vector space, its Frobenius twist $W^{\xi}$ is $W$ with new scalar multiplication given by $\lambda \cdot x^{\xi} = (\lambda^p x)^{\xi}$ for all $\lambda \in k$ and $x \in W$.

\begin{defn} Let $\Phi: \M(k) \ra \M(k)$ be defined as follows.

If $p=2$, let $
(\Phi M)_n \simeq
\begin{cases}
M_m^{\xi} & \text{if } n = 2m \\ 0 & \text{otherwise}.
\end{cases}$

If $p$ is odd, let $
(\Phi M)_n \simeq
\begin{cases}
M_{2m}^{\xi} & \text{if } n = 2pm \\ 0 & \text{otherwise}.
\end{cases}$
\end{defn}

The following lemma is easily checked.
\begin{lem} $\Phi$ is a functor of symmetric tensor categories. In particular, $\Phi$ is exact, and there are natural isomorphisms $\Phi(M \otimes N) \simeq \Phi M \otimes \Phi N$ and $\Phi (M^{\vee}) \simeq (\Phi M)^{\vee}$.
\end{lem}

The next lemma is less obvious, but also not hard to verify. (Compare with \cite[\S1.1.1]{glm92}.)

\begin{lem} \label{Tate lemma} Let $K(M)$ and $C(M)$ respectively be the kernel and cokernel of the norm map
$N_{\Sigma_p}: (M^{\otimes p})_{\Sigma_p} \ra (M^{\otimes p})^{\Sigma_p}$ from the $\Sigma_p$--coinvariants to the $\Sigma_p$--invariants of $M^{\otimes p}$.
There are natural isomorphisms
$$ K(M) \simeq \Phi M \simeq C(M),$$
and thus natural maps $\Phi M \hra (M^{\otimes p})_{\Sigma_p}$ and $(M^{\otimes p})^{\Sigma_p} \ora \Phi M$.
\end{lem}

The following definitions are easily seen to agree with Definitions \ref{F module defn} and \ref{V module defn}.
\begin{defns} The categories $F$--$\M(k)$ and $V$--$\M(k)$ are defined as follows. \\

\noindent{\bf (a)} An $F$--module is a reduced $M \in \M(k)$ equipped with an $\M(k)$-morphism $F: \Phi M \ra M$. An $F$--module morphism is a $\M(k)$--morphism $f: M \ra N$ such that $f \circ F = F \circ \Phi(f): \Phi M \ra N$. \\

\noindent{\bf (b)} An $V$--module is a reduced $M \in \M(k)$ equipped with an $\M(k)$-morphism $V: M \ra \Phi M$. An $F$--module morphism is a $\M(k)$--morphism $f: M \ra N$ such that $V \circ f = \Phi(f) \circ V: M \ra \Phi N$.
\end{defns}

We note that the properties of $\Phi$ show that $F$--$\M(k)$ and $V$--$\M(k)$ inherit the structure of symmetric $k$--linear tensor categories from $\M(k)$.

\begin{defns} Forgetful functors $\A^*(k) \ra F\text{--}\M(k)$ and $\C_*(k) \ra V\text{--}\M(k)$ are defined as follows. \\

\noindent{\bf (a)} If $A$ is a commutative algebra in $\M(k)$, the composite
$$ \Phi A \hra (A^{\otimes p})_{\Sigma_p} \xra{\nabla} A$$
gives $A$ the structure of an $F$--module. \\

\noindent{\bf (b)} If $C$ is a cocommutative algebra in $\M(k)$, the composite
$$ C \xra{\Delta} (C^{\otimes p})^{\Sigma_p} \ora \Phi C$$
gives $C$ the structure of an $V$--module.
\end{defns}

\subsection{The classification of $F$--modules when $k$ is perfect}

We continue to let $k$ be a field of characteristic $p$.

\begin{defn} \label{F module defn 2} We define indecomposable $F$--modules $N(n,j)$ as follows. \\

\noindent{\bf (a)} Let $p=2$. For every $n\geq 1$ and $0 \leq j \leq \infty$, we let $N(n,j)$ have a basis $x_0, \dots, x_j$ with $|x_i|=2^in$ and with $F(x_i)=x_{i+1}$ for $i\geq 0$. \\

\noindent{\bf (b)} Let $p$ be odd.  For every $m\geq 1$ and $0 \leq j \leq \infty$, we let
$N(2m,j)$ have a basis  $y_0, \dots, y_j$ with $|y_i|=2p^im$ and with $F(y_i)=y_{i+1}$ for $i\geq 0$.  For every $m \geq 0$, we let $N(2m+1,0)$ be one dimensional with generator in degree $2m+1$ (with trivial $F$--module structure, of course).
\end{defn}

We have a classification theorem.

\begin{thm} \label{F module class thm}  If $k$ is a perfect field of characteristic $p$, then every $F$--module can be written uniquely as a direct sum of $F$--modules of the form $N(n,j)$.
\end{thm}

We also have the following classification of projectives and injectives.

\begin{prop} \label{F module prop}  If $k$ is perfect, indecomposable projectives and injectives in $F$--$\M(k)$ are as follows. \\

\noindent{\bf (a)} If $p=2$, the $F$--modules $N(n,\infty)$ are projective. The $F$--modules $N(n,j)$ with $n$ odd are injective. \\

\noindent{\bf (b)} If $p$ is odd, the $F$--modules $N(2m,\infty)$ and $N(2m+1,0)$ are projective. The $F$--modules $N(2m,j)$ with $p \nmid m$ and $N(2m+1,0)$ are injective.
\end{prop}

We prove both the theorem and proposition together.

First assume $p=2$.  Any $F$--module $N$ canonically decomposes as a direct sum of $F$--modules
$$ N \simeq \bigoplus_{n \text{ odd}} N(n)$$
where $N(n) = \bigoplus_{i=0}^{\infty} N_{2^in}$.

By regrading the vector spaces in $N(n)$ -- view $N_{2^in}$ as having grading $i$ -- each $N(n)$ is in the category
$\mathcal N(k)$ consisting of sequences of $k$--vector spaces $(N_0,N_1,\dots)$ equipped with $k$--linear maps $(N_i)^{\xi} \ra N_{i+1}$.

Since the field $k$ is perfect, we can `untwist' our vector spaces, and conclude that $\mathcal N(k)$ is equivalent to the category of sequences of $k$--vector spaces $(N_0,N_1,\dots)$ equipped with $k$--linear maps $N_i \ra N_{i+1}$.

But this last category is equivalent to the category of non-negatively graded $k[t]$--modules of finite type, where $t$ has grading 1.  We thus have described equivalences of abelian categories
$$ F\text{--}\M(k) \simeq \prod_{n \text{ odd}} \mathcal N(k) \simeq \prod_{n \text{ odd}} k[t]\text{--modules}.$$

Since $k[t]$ is a graded PID, its modules of finite type can be written uniquely as the direct sum of cyclic modules, and these cyclic modules have the form $k[t]t^i$ and $k[t]t^i/k[t]t^{i+j+1}$ for $j\geq 0$.  It is not hard to check that these are injective precisely when $i=0$, and, more obviously, that the modules $k[t]t^i$ are projective.

Under the equivalence
$\displaystyle  F\text{--}\M(k) \simeq \prod_{n \text{ odd}} k[t]\text{--modules}$,
the modules $k[t]t^i$ and $k[t]t^i/k[t]t^{i+j+1}$ in the $n$th component of the product will correspond to the $F$--modules $N(2^in, \infty)$ and $N(2^in, j)$.

The proofs of the theorem and proposition when $p$ is odd is similar.  Now one has a decomposition of $F$--modules
$$ N \simeq \bigoplus_{m \text{ with } p\nmid m} N(2m) \oplus \bigoplus_{m} N_{2m+1}$$
where $N(2m) = \bigoplus_{i=0}^{\infty} N_{2p^im}$.

As before, this leads to an equivalence of abelian categories
$$ F\text{--}\M(k) \simeq \prod_{m \text{ with } p\nmid m} k[t]\text{--modules} \times \prod_m k\text{--vector spaces}.$$
The modules $k[t]t^i$ and $k[t]t^i/k[t]t^{i+j+1}$ in the $m$th component of the first infinite product will correspond to the $F$--modules $N(2p^im, \infty)$ and $N(2p^im, j)$, while the vector space $k$ in the $m$th component of the second infinite product will correspond to the $F$--module $N(2m+1,0)$.

\subsection{The classification of $V$--modules when $k$ is perfect}

The results of the last subsection give us results about $V$--modules using duality.

\begin{defn} We define indecomposable $V$--modules $M(n,j)$ as follows. \\

\noindent{\bf (a)} Let $p=2$. For every $n \geq 1$ and $0 \leq j \leq \infty$, we let $M(n,j)$ have a basis $x_0, \dots, x_j$ with $|x_i|=2^in$ and with $V(x_{i+1})=x_{i}$ for $i\geq 0$. \\

\noindent{\bf (b)} Let $p$ be odd.  For every $m \geq 1$ and $0 \leq j \leq \infty$, we let
$M(2m,j)$ have a basis  $y_0, \dots, y_j$ with $|y_i|=2p^im$ and with $V(y_{i+1})=y_{i}$ for $i\geq 0$. For every $m \geq 0$, we let $M(2m+1,0)$ be one dimensional with generator in degree $2m+1$.
\end{defn}

\begin{lem} \label{Phi lem}  There are isomorphisms of $V$--modules as follows. \\

\noindent{\bf (a)} $\Phi M(n,j) \simeq M(pn,j)$ when $p=2$, and when $p$ is odd and $n$ is even. \\

\noindent{\bf (b)} $\displaystyle \colim_{j\ra \infty} M(n,j) \simeq M(n, \infty)$.
\end{lem}

\thmref{F module class thm} and duality implies the next theorem.

\begin{thm} \label{V module class thm}  If $k$ is a perfect field of characteristic $p$, then every $V$--module can be written uniquely as a direct sum of $V$--modules of the form $M(n,j)$.
\end{thm}

Similarly, \propref{F module prop} implies the next proposition.

\begin{prop} \label{V module prop}  If $k$ is perfect, indecomposable projectives and injectives in $V$--$\M(k)$ are as follows. \\

\noindent{\bf (a)} If $p=2$, the $V$--modules $M(n,\infty)$ are injective. The $V$--modules $M(n,j)$ with $n$ odd are projective. \\

\noindent{\bf (b)} If $p$ is odd, the $V$--modules $M(2m,\infty)$ and $M(2m+1,0)$ are injective. The $V$--modules $M(2m,j)$ with $p \nmid m$ and $M(2m+1,0)$ are projective.
\end{prop}

It will be technically useful to also consider the $M(n,j)$ in the category $V$--modules whose underlying graded vector space structure is not necessarily of finite type.

\begin{prop} \label{unbounded proj prop} Let $V$--$\M(k)^u$ be the category of $V$--modules that are not necessarily of finite type. \\

\noindent{\bf (a)} If $p=2$, the $V$--module $M(n,j)$ with $n$ odd and $j<\infty$ are projective generators for $V$--$\M(k)^u$, as there are a natural isomorphisms
$$  \Hom_{V\text{--}\M(k)^u}(M(n,j), M) \simeq M_{2^jn}.$$

\noindent{\bf (b)} If $p$ is odd, $V$--modules $M(2m,j)$ with $p \nmid m$ and $j<\infty$, together with the $V$--modules $M(2m+1,0)$, are projective generators  for $V$--$\M(k)^u$, as there are a natural isomorphisms
$$  \Hom_{V\text{--}\M(k)^u}(M(2m,j), M) \simeq M_{2p^jn},$$
and
$$  \Hom_{V\text{--}\M(k)^u}(M(2m+1,0), M) \simeq M_{2m+1}.$$
\end{prop}

\subsection{The category $\Ho_*(k)$}  We record some useful properties of the category of connected cocommutative Hopf algebras in $\M(k)$.

\begin{prop} \label{Hopf alg properties prop} The category $\Ho_*(k)$ satisfies the following properties. \\

\noindent{\bf (a)} $\Ho_*(k)$ is pointed with initial/terminal object $k$. \\

\noindent{\bf (b)} $\Ho_*(k)$ has pullbacks, and thus kernels and finite products. \\

\noindent{\bf (c)} $\Ho_*(k)$ has coproducts, subject to our finite type condition: a family $\{H_i\}$ in $\Ho_*(k)$, such that, for any $c$, only a finite number of the $\bar H_i$ are not $c$--connected, admits a coproduct $\underset{i}{*}H_i$.
\end{prop}

These properties are all well-known, and are all implicitly or explicitly in \cite{milnor moore} or \cite{moore smith}. See also \cite[\S 1.1.2]{glm92}

Concerning property (b), $\Ho_*(k)$ is the category of group objects in $\C_*(k)$, the category of connected cocommutative coalgebras in $\M(k)$ \cite[\S 8]{milnor moore}.  It follows that pullbacks in $\Ho_*(k)$ are `induced' from pullbacks in $\C_*(k)$. For example, given $H_1, H_2 \in \Ho_*(k)$, $H_1 \otimes H_2$ is the categorical product.

Concerning (c), given the family $\{H_i\}$ in $\Ho_*(k)$, let $\underset{i}{*}H_i$ be the coproduct of the $H_i$, just viewed as graded algebras.  The family of algebra maps
$$ \Delta_j: H_j \xra{\Delta} H_j \otimes H_j \ra \underset{i}{*}H_i \otimes \underset{i}{*}H_i$$
then defines an algebra map $\Delta: \underset{i}{*}H_i \ra \underset{i}{*}H_i \otimes \underset{i}{*}H_i$ giving $\underset{i}{*}H_i$ the structure of a cocommutative Hopf algebra, and it is straightforward to verify that $\underset{i}{*}H_i$ is the coproduct in $\Ho_*(k)$ of the $H_i$.  Compare with \cite[Proposition 2.4]{moore smith}. \\

The next proposition will be used in the proof of \thmref{split thm}.

\begin{prop}  \label{Hopf alg kernel prop} Given $f: H_1 \ra H_2$ in $\Ho_*(k)$, let $I \subset H_1$ be the kernel of $f$ when $f$ is just viewed as a map of algebras, and let $K \subset H_1$ be the Hopf algebra kernel of $f$.  Then $I = (\bar K)$, the two-sided ideal generated by $\bar K$.
\end{prop}

This is Proposition 1.3 of \cite{wilkerson}, and Wilkerson notes that this is essentially \cite[Proposition 4.9]{milnor moore}.  (Versions of this are also proved in the nongraded setting in \cite[Theorem 3.1]{newman75} and \cite[Lemma 16.0.2]{sweedler}. Neither of these papers cite \cite{milnor moore}; indeed, Sweedler's book has no references at all.) \\

Finally we note a couple of easily verified results about the behavior of $\Phi$ and $Q$.  Regarding the first of these, it is illuminating to note that $Q: \Ho_*(k) \ra V\text{--}\M(k)$ has a right adjoint \cite[Proposition 1.1.2.4]{glm92}.

\begin{prop} \label{Q prop}  Let $k$ be a field of positive characteristic. \\

\noindent{\bf (a)} The natural $V$--module map $\bigoplus_i Q(H_i) \ra Q(\underset{i}{*} H_i)$ is an isomorphism. \\

\noindent{\bf (b)} If $H \in \Ho_*(k)$ then so is $\Phi H$, and the natural map $V: H \ra \Phi H$ is a map of Hopf algebras. \\

\noindent{\bf (c)} There is a natural $V$--module isomorphism $Q(\Phi H) \simeq \Phi (QH)$.
\end{prop}

\section{Free split Hopf algebras, and the proof of Theorems \ref{split thm} and \ref{lifting thm}} \label{split thm proof sec}

Thoughout this section, let $k$ be a perfect field of characteristic $p$.

\subsection{The work of Goerss, Lannes, and Morel}

Let $\Ho_*(k)^u$ be the category of connected cocommutative Hopf algebras in graded $k$--vector spaces which are not necessarily of finite type.

When $k$ is the prime field $\F_p$, \cite[Th\'eor\`eme 1.2.1]{glm92} says the following.

\begin{thm} \label{glm thm}  Let $M$ be a projective object in $V$--$\M(k)^u$.  \\

\noindent{\bf (a)} The induced $V$--module structure on the free associative algebra $T(M) = \bigoplus_{k=0}^{\infty} M^{\otimes k}$ extends to a cocommutative Hopf algebra structure,  and the resulting Hopf algebra will be a projective object in $\Ho_*(k)^u$.  The Hopf algebras arising from any two such extensions are isomorphic. \\

\noindent{\bf (b)} If $P$ is a projective object in $\Ho_*(k)^u$, then $QP$ will be projective in $V$--$\M(k)^u$, and $P$ will be free as an algebra.
\end{thm}

In the proof of this, the last statement of part (a) and the last part of part (b) follow from \cite[Lemme 1.2.3.3]{glm92} and  \cite[Lemme 1.2.3.2]{glm92} which read as follows.

\begin{lem} \label{glm lem}  Let $P$ be a projective object in $\Ho_*(k)^u$.  Given $H \in \Ho_*(k)^u$, and a $V$--module map $g: QP \ra QH$, there exists a morphism of Hopf algebras $G: P \ra H$ such that $QG = g$.
\end{lem}

\begin{lem} \label{tensor lem} Let $f: A \rightarrow B$ be a morphism of graded associative connected $k$--algebras, with $B$ a free algebra.  If $Qf: QA \ra QB$ is a monomorphism, so is $f$. If $QF$ is an isomorphism, so is $f$.
\end{lem}

The proof in \cite{glm92} of this last lemma works for all fields $k$, not just prime fields.
This is also true for the theorem and the other lemma, but verifying this takes some work.  Alternatively, one can deduce the theorems for a general perfect field $k$ from the prime field case, as follows.

Tensoring with $k$ over $\F_p$ induces functors
$$ k \otimes_{\F_p}\text{\underline{\hspace{.1in}}}: \Ho_*(\F_p)^u \ra \Ho_*(k)^u$$
and
$$k \otimes_{\F_p}\text{\underline{\hspace{.1in}}}: V\text{--}\M(\F_p)^u \ra V\text{--}\M(k)^u$$
which are left adjoint to forgetful functors.

Formal arguments show that these functors will send projectives to projectives, and the second of these functors sends the projective generators of $V\text{--}\M(\F_p)^u$ to the projective generators of  $V\text{--}\M(k)^u$.  It is then straightforward to deduce \thmref{glm thm} and \lemref{glm lem} for a general perfect field $k$ from the case when $k=\F_p$.

\begin{addendum} \label{thm addendum} \thmref{glm thm} and \lemref{glm lem} also hold with the categories $V\text{--}\M(k)^u$ and $\Ho_*(k)^u$ replaced by $V\text{--}\M(k)$ and $\Ho_*(k)$.
\end{addendum}
\begin{proof} The point is that there are projectives in $V\text{--}\M(k)$ that are not projective in $V\text{--}\M(k)^u$, namely the projectives that have direct summand factors of the form $M(n,\infty)$.

To work with these, we proceed as follows.  Given an $M(n,j)$ with $j<\infty$ and projective (i.e. with $n$ odd or with $p$ odd and $p\nmid n$), let $H(n,j)$ be the projective Hopf algebra that one gets by applying \thmref{glm thm}(a).  Then \lemref{glm lem} tells us that the inclusion $M(n,j) \hra M(n,j+1)$ will be `covered' by a Hopf algebra map $H(n,j) \ra H(n,j+1)$, and this will also be an inclusion thanks to \lemref{tensor lem}. Now define $H(n,\infty) \in \Ho_*(k)$ to be the union of these.

Using that $\displaystyle \colim_{j \ra \infty} H(n,j) = H(n,\infty)$, one sees that $H(n,\infty)$ will be projective in $\Ho_*(k)$ as follows.  Given a diagram in $\Ho_*(k)$
$
\SelectTips{cm}{}
\xymatrix{
& H_1 \ar[d]^q  \\
H(n,\infty) \ar[r]^f & H_2 }
$

\noindent with $q$ surjective, let $S(j)$ be the set of lifts
$
\SelectTips{cm}{}
\xymatrix{
&& H_1 \ar[d]^q  \\
H(n,j) \ar[r] \ar@{..>}[rru] &H(n,\infty) \ar[r]^f & H_2. }
$

\noindent Then $S(0) \la S(1) \la S(2) \la \dots$ will be an inverse system of nonempty finite sets: nonempty since $H(n,j)$ is projective, and finite since $H_1$ has finite type and $H(n,j)$ is finitely generated as an algebra.  Thus $\displaystyle \lim_j S(j)$ will be nonempty, and any element of this inverse limit will be a lift of $f$.

Similarly, any $P \in \Ho_*(k)$ which is the coproduct of these $H(n,j)$'s will be projective, and for any such $P$, \lemref{glm lem} will apply, as long as $H$ is in $\Ho_*(k)$ and not just $\Ho_*(k)^u$.
\end{proof}

\begin{ex}  When $p=2$, $M(1,0) = \langle x \rangle$, $M(1,1) = \langle x,y \rangle$, and $M(1,2) = \langle x,y, z \rangle$, with $|x|=1$, $|y|=2$, and $|z|=4$, with $V(y)=x$ and $V(z)=y$.

Corresponding Hopf algebras are tensor algebras $H(1,0) = T(x)$, $H(1,1) = T(x,y)$, and $H(1,2) = T(x,y,z)$, with
$$ \Delta(x) = x \otimes 1 + 1 \otimes x,$$
$$ \Delta(y) = y \otimes 1 + x \otimes x + 1 \otimes y,$$
$$ \Delta(z) = z \otimes 1 +  xy \otimes x + x^3 \otimes x + + y \otimes y + x \otimes x^3 + x \otimes xy + 1 \otimes z.$$
\end{ex}

\begin{rem}  The authors of \cite{glm92} refer to the Hopf algebras $H(n,j)$ as algebras of `noncommutative Witt vectors', as the algebras of classical Witt vectors occur as their bicommutative quotients.  See \cite[\S 1.4]{glm92} for a nice discussion of this.
\end{rem}

\subsection{Proofs of \thmref{split thm} and \thmref{lifting thm}}

Note that \thmref{glm thm}, as enhanced by \adref{thm addendum}, implies \thmref{split thm} in the case when $M$ is a projective $V$--module: given such an $M$, there is a unique Hopf algebra $H(M) \in \Ho_*(k)$ such that $H(M)$ is free, and $QH(M) \simeq M$ as $V$--modules.  Since \thmref{glm thm} tells us that projective objects in $\Ho_*(k)$ are precisely the Hopf algebras $H(M)$ with $M$ projective, one sees that   \thmref{lifting thm} for such $M$ is just \lemref{glm lem}, as enhanced by \adref{thm addendum}.

Furthermore, if $\displaystyle M = \bigoplus_i M(n_i, j_i)$, with each $M(n_i,j_i)$ projective, then $H(M) \simeq \underset{i}{*} H(n_i,j_i)$, and this coproduct decomposition is unique.

Our goal now is to extend these results to arbitrary $M \in V\text{--}\M(k)$.

It is easy to define our Hopf algebras $H(M)$.  Recall that we have already defined $H(n,j)$, free as an algebra and with $QH(n,j) \simeq M(n,j)$, when $M(n,j)$ is projective, i.e. when $n$ is odd or when $p$ is odd and $n=2m$ with $p\nmid m$.

\begin{defns} Given $M \in V\text{--}\M(k)$, we define $H(M) \in \Ho_*(k)$ as follows. \\

\noindent{\bf (a)} If $p=2$ and $m$ is odd, let $H(2^im,j) = \Phi^i H(m,j)$.  If $p$ is odd, and $p\nmid m$, let $H(2p^im,j) = \Phi^i H(2m,j)$. \\

\noindent{\bf (b)} If $\displaystyle M = \bigoplus_i M(n_i, j_i)$, we let $H(M) = \underset{i}{*} H(n_i,j_i)$.
\end{defns}

\begin{lem} $H(M)$ is free as an algebra, $QH(M) \simeq M$ as $V$--modules, and $H(M)$ is split: i.e. $\overline{H(M)} \ra M$ has a section.
\end{lem}
\begin{proof} The first two parts are evident by construction. $\overline{H(M)} \ra M$ obviously has a section if $M$ is projective.  Applying $\Phi^i$ to such sections when $M = M(m,j)$ is projective (with $m$ even if $p$ is odd), defines a section for $\overline{H(p^im,j)} \ra M(p^im,j)$.  Finally, given a family of $V$--module sections $s_i: M(i) \ra \overline{H(M(i))}$ to $\overline{H(M(i))} \ra M(i)$, the composite
$$ \bigoplus_i M(i) \xra{\oplus_i s_i} \bigoplus_i \overline{H(M(i))} \ra \underset{i}{*} H(M(i))$$
shows that $H(\bigoplus_i M(i)) = \underset{i}{*} H(M(i))$ will also be split.
\end{proof}

To finish the proof of \thmref{split thm}, we need show that $H(M)$ is uniquely defined by the properties listed in this last lemma.   But first we prove \thmref{lifting thm}, which should be viewed as appropriate version of \lemref{glm lem} with $H(M)$ replacing the projective $P$ in the statement.

We remind readers of what \thmref{lifting thm} asserts: \\

{\em Let $M \in V\text{--}\M(k)$, and $H \in \Ho_*(k)$.  Given a $V$--module morphism $g: M \ra Q(H)$ that admits a lift to $\bar H$, there exists $G: H(M) \ra H$ in $\Ho_*(k)$ such that $QG = g: M \ra QH$.} \\

Note that the lifting hypothesis always holds if $M$ is a projective $V$--module, or if $H$ is a split Hopf algebra.

\begin{ex} Here is an example that perhaps will help readers appreciate this theorem, and aid in following the proof.

Let $k = \F_2$. Let $M=M(6,0)$ so that $H(M) = H(6,0) = T(z)$ with $z$ a primitive generator of degree 6.  Then Hopf algebra morphisms $H(6,0) \ra H$ correspond to primitives in $H_6$. Meanwhile, $V$--module maps $M(6,0) \ra \bar H$ correspond to elements $y \in H_6$ with $V(y)=0$.

Now let $H = H_*(\Omega \Sigma \mathbb CP^3;\F_2) = T(y_1,y_2,y_3)$, with $|y_i| = 2i$ and with
$$ \Delta(y_1) = y_1 \otimes 1 + 1 \otimes y_1,$$
$$ \Delta(y_2) = y_2 \otimes 1 + y_1 \otimes y_1 + 1 \otimes y_2,$$
$$ \Delta(y_3) = y_3 \otimes 1 +  y_2 \otimes y_1 + y_1 \otimes y_2 + 1 \otimes y_3.$$

The element $y_3$ satisfies $V(y_3) = 0$, but is not primitive.  \thmref{lifting thm} now guarantees that there exists another $y \in H_6$ that {\em is} primitive, with $y \equiv y_3 \mod \bar H^2$.  A little bit of fiddling reveals that $y = y_3 + y_1y_2 + y_1^3$ does the job.
\end{ex}

\begin{proof}[Proof of \thmref{lifting thm}]  It certainly suffices to prove the theorem when $M = M(n,j)$.  Furthermore, we already know the theorem is true when $M=M(n,j)$ is a projective $V$--module.

We are left having to prove the theorem when $M=M(p^in,j)$ with $i\geq 1$ and $0 \leq j \leq \infty$, in the cases
$
\begin{cases}
n \text{ is odd} & \text{if } p=2 \\ n \text{ is even and }p\nmid n & \text{if } p \text{ is odd}.
\end{cases}
$ \\
In these cases, $M(n,j)$ is a projective $V$--module.

It also suffices to assume that $j$ is finite, since $\displaystyle M(p^in, \infty) = \colim_{j \ra \infty} M(p^in,j)$ and $\displaystyle H(p^in, \infty) = \colim_{j \ra \infty} H(p^in,j)$.

Now suppose we are given a $V$--module map $g: M(p^in,j) \ra QH$ and a lifting of this $\tilde g: M(p^in,j) \ra \bar H$.  Our goal is to show that there is a map of Hopf algebras $G: H(p^in,j) \ra H$ with $QG=g$.

As a $V$--module, $M(p^in,j)$ is generated by its top class $x_j$ in degree $p^{i+j}n$, and $V^{j+1}(x_j) = 0$.  So $\tilde g$ corresponds to an element $x \in \ker V^{j+1}$ of degree $p^{i+j}n$.

Let $H[j+1]$ be the Hopf algebra kernel of $V^{j+1}: H \ra \Phi^{j+1}H$. By \propref{Hopf alg kernel prop},  $\ker(V^{j+1}) = (\bar H[j+1])$, so we can write $x = x^{\prime} + yz$, with $x^{\prime},y \in \bar H[j+1]$ and $z \in \bar H$.  Since $x^{\prime} \equiv x \mod \bar H^2$, $x^{\prime}$ will also correspond to a $V$--module map $\tilde g^{\prime}: M(p^in,j) \ra \bar H$ that is a lift of $g$, but now has image in $\bar H[j+1]$.

The $V$--module $M(p^in,j)$ fits into a short exact sequence
$$ 0 \ra M(n,i-1) \xra{\iota} M(n, i+j) \xra{\rho} M(p^in,j).$$
This induces maps of Hopf algebras
$$ H(n,i-1) \xra{\iota} H(n, i+j) \xra{\rho} H(p^in,j)$$
such that $\iota$ is 1-1, and $\rho$ induces an isomorphism
$$H(n,i+j)//H(n,i-1) \simeq H(p^in,j).$$
Note also that $H(n,i-1)$ is generated by elements in the image of $V^{j+1}$.

As $M(n,i+j)$ is projective, we can  apply the already proved $M(n,i+j)$ case of the theorem to the composite
$$M(n,i+j) \xra{\rho} M(p^in,j) \xra{\tilde g^{\prime}} \bar H[j+1]$$
to get a Hopf algebra map $G^{\prime}: H(n,i+j) \ra H$ factoring through $H[j+1]$, with $QG^{\prime} = g \circ \rho$.

As $G^{\prime}$ factors through $H[j+1]$, it will vanish on the generators of $H(n,i-1)$, and so $G^{\prime}$ will factor through $H(m,i+j)//H(n,i-1) = H(p^in,j)$.  The resulting map of Hopf algebras $G: H(n,j) \ra H$ will satisfy $QG = g$.
\end{proof}

The remaining part of \thmref{split thm} is now easily proved.  Suppose that $H \in \Ho_*(k)$ is split and free as an algebra.  If $g: M \xra{\sim} QH$ is a $V$--module isomorphism, we need to show that $H(M) \simeq H$.  \thmref{lifting thm} implies that there exists a map of Hopf algebras $G: H(M) \ra H$ such that $QG = g$.  But then \lemref{tensor lem} implies that $G$ is then also an isomorphism. \\

In proving \thmref{split thm}, we have also established a classification theorem.

\begin{thm} \label{classification thm}  Hopf algebras in $\Ho_*(k)$ that are split, and free as algebras, are precisely the Hopf algebras of the form $\underset{i}{*} H(n_i,j_i)$.
\end{thm}

\section{Details of our cautionary example, \exref{bad example}} \label{bad example sec}

Recall the definitions of the two Hopf algebras in \exref{bad example}.

We work over $k=\F_2$.  As algebras $H_1 = T(x,y,z)= H_2$, with $|x|=1$, $|y|=2$, and $|z|=4$.

The coproduct structure is given as follows.  The classes $x,y$ are primitive in both $H_1$ and $H_2$, while the coproducts $\Delta_1$ and $\Delta_2$ act as follows on $z$:
$$\Delta_1(z) = z\otimes 1 + y \otimes y + 1 \otimes z.$$
and
$$\Delta_2(z) = z\otimes 1 + x^2 \otimes x^2 + 1 \otimes z.$$

\subsection{Algebraic details}  We check that $H_1$ and $H_2$ are isomorphic as both algebras and coalgebras, but not as Hopf algebras.

Obviously $H_1$ and $H_2$ are isomorphic as algebras.

To see that $H_1$ and $H_2$ are not isomorphic as Hopf algebras, we just note that the $V$--modules $QH_1$ and $QH_2$ are not isomorphic: $QH_1 = \langle x,y,z\rangle = QH_2$, with with $V(z) = y$ in $QH_1$, and $V = 0$ in $QH_2$.

To see that $H_1$ and $H_2$ are isomorphic as coalgebras, we describe an explicit isomorphism.  Any element in the monomial basis for $T(x,y,z)$ can be written in the form
$$ (x^{2i_1+\epsilon_1}y^{j_1}z^{k_1})(x^{2i_2+\epsilon_2}y^{j_2}z^{k_2}) \cdots (x^{2i_r+\epsilon_r}y^{j_r}z^{k_r})$$
with $\epsilon_a$ equal to 0 or 1. We write this as
$\displaystyle \prod_{a=1}^r x^{2i_a+\epsilon_a}y^{j_a}z^{k_a}$.

Then a coalgebra isomorphism $\Theta: H_1 \ra H_2$ is given by
$$ \Theta(\prod_{a=1}^r x^{2i_a+\epsilon_a}y^{j_a}z^{k_a}) = \prod_{a=1}^r x^{2j_a+\epsilon_a}y^{i_a}z^{k_a}.$$
Thus, for example, $\Theta(xyz)=x^3z$ and $\Theta(xy^2z^3x^4y^5z^6) = x^5z^3x^{10}y^2z^6$.

In checking that this is a coalgebra morphism,  it is illuminating to note that $H = T(x^2,y,z)$ is a subHopf algebra of both $H_1$ and $H_2$, and that $\Theta$ induces the Hopf algebra automorphism  on $H$ sending $x^2,y,z$ to $y,x^2,z$.

\subsection{Topological realization details}

We check that $H_1$ and $H_2$ can be topologically realized.

This is clear for $H_1$ as $H_1 \simeq H_*(\Omega \Sigma (S^1\vee \mathbb C\mathbb P^2);\F_2)$,
with $x,y,z$ corresponding to the nonzero homogeneous elements in $\widetilde H_*(S^1 \vee \mathbb C\mathbb P^2);\F_2)$.

Now we check that $H_2 \simeq H_*(\Omega (S^3\vee C);\F_2)$,
where  $C$ is the mapping cone of the composite
$S^4 \xra{\Sigma \eta} S^3 \xra{\eta} S^2$,
and $\eta: S^3 \ra S^2$ is the Hopf map.

By \cite[Corollary 3.2]{berstein}, the two inclusions $\Omega S^3 \hra \Omega(S^3\vee C) \hookleftarrow \Omega C$ induce a Hopf algebra isomorphism
$$ H_*(\Omega S^3; \F_2) * H_*(\Omega C;\F_2) \xra{\sim} H_*(\Omega (S^3\vee C);\F_2).$$
Furthermore, $H_*(\Omega S^3; \F_2) = T(y)$ with $|y|=2$.

It thus suffices to show the following.

\begin{prop} \label{H_*C prop} As a Hopf algebra, $H_*(\Omega C;\F_2) \simeq T(x,z)$ with $|x|=1$, $|z|=4$ and $\Delta(z) = z \otimes 1+ x^2 \otimes x^2 + 1 \otimes z$.
\end{prop}

\begin{proof}  Recall that if $X$ is simply connected, the Eilenberg-Moore spectral sequence converges to $H_*(\Omega X;\F_2)$ with $E^1$--term equal to the tensor algebra on $\Sigma^{-1}\tilde H_*(X;\F_2)$, and this spectral sequence collapses if $X$ is a suspension.

The commutative square
\begin{equation*}
\SelectTips{cm}{}
\xymatrix{
S^4  \ar@{=}[d] \ar[r]^{\Sigma \eta} & S^3 \ar[d]^{\eta}  \\
S^4 \ar[r]^{\eta \circ \Sigma \eta} & S^2  }
\end{equation*}
induces a map between cofibration sequences
\begin{equation*}
\SelectTips{cm}{}
\xymatrix{
S^4  \ar@{=}[d] \ar[r]^{\Sigma \eta} & S^3 \ar[d]^{\eta} \ar[r]^-i & \Sigma \mathbb C P^2 \ar[d]^f  \\
S^4 \ar[r]^{\eta \circ \Sigma \eta} & S^2 \ar[r]^j & C. }
\end{equation*}
The right square of this diagram is a homotopy pushout, and we see that there is a short exact sequence
$$ 0 \ra \widetilde H_*(S^3;\F_2) \xra{i_*} \widetilde H_*(S^2 \vee \Sigma \mathbb C P^2;\F_2) \xra{(j\vee f)_*} \widetilde H_*(C;\F_2) \ra 0).$$
Since $(j \vee f)_*$ is onto, and the EMSS computing $H_*(\Omega(S^2 \vee \Sigma \mathbb C P^2); \F_2)$ collapses, the same is true for the EMSS computing $H_*(\Omega C; \F_2)$. We conclude that there is an epimorphism of Hopf algebras
$$ H_*(\Omega \Sigma (S^1 \vee \mathbb C P^2;\F_2) \ra H_*(\Omega C;\F_2).$$
Furthermore, if we let $a \in H_1(\Omega S^2;\F_2)$ and $c \in H_4(\mathbb C P^2;\F_2)$ be the nonzero classes, and $\tilde f: \mathbb C P^2 \ra \Omega C$ be the adjoint to $f$, we see that $H_*(\Omega C;\F_2) = T(x,z)$ where $x=(\Omega j)_*(a)$ and $z = \tilde f_*(c)$.

It remains to show that $\Delta(z) = z \otimes 1 + x^2 \otimes x^2 + 1 \otimes z$. To show this we consider the adjoint of the right square in the previous diagram:
\begin{equation*}
\SelectTips{cm}{}
\xymatrix{
 S^2 \ar[d]^{\tilde \eta} \ar[r]^-i & \mathbb C P^2 \ar[d]^{\tilde f}  \\
 \Omega S^2 \ar[r]^{\Omega j} & \Omega C. }
\end{equation*}
Let $b \in H_2(S^2;\F_2)$ be the nonzero element.  Then $\Delta(c) = c \otimes 1 + i_*(b) \otimes i_*(b) + 1 \otimes c$. Meanwhile, since $\eta$ is the Hopf map (and thus has Hopf invariant 1!), we see that $\tilde \eta_*(b) = a^2$.  Thus $\tilde f_*(i_*(b)) = x^2$, so $\Delta(z) = z \otimes 1 + x^2 \otimes x^2 + 1\otimes z$ as needed.
\end{proof}

\section{Quasi--shuffle algebras} \label{quasi shuffle sec}

In this section, we offer a short discussion of the Hopf algebras $J(C)$, for $C \in \C_*(k)$, and $J^{\vee}(A)$, for $A \in \A_*(k)$.

\subsection{The free Hopf algebra $J(C)$}

Recall that $J: \C_*(k) \ra \Ho_*(k)$ is defined as the left adjoint to the forgetful functor.  From this definition, it is quite easy to see that, as an algebra, $J(C)$ will be the tensor algebra $\displaystyle T(C) = \bigoplus_{k=0}^{\infty} \bar C^{\otimes k}$.
Using that $T(C)$ is the free algebra generated by $\bar C$, the coproduct $\Delta_{J(C)}$ on $J(C)$ is then induced by the coproduct $\Delta_C$ on $C$: $\Delta_{J(C)}$ is the algebra map
$$ T(C) \ra T(C) \otimes T(C)$$
induced by the composite
$$ \bar C \xra{\bar \Delta_C} \overline{(C\otimes C)} = (\bar C\otimes k) \oplus (\bar C \otimes \bar C) \oplus (k \otimes C) \hra T(C) \otimes T(C).$$

It is not hard to track the components of $\Delta_{J(C)}$. The component
$$ \Delta^n_{l,m}: \bar C^{\otimes n} \ra \bar C^{\otimes l} \otimes \bar C^{\otimes m}$$
will be the part of the composite
$$ \bar C^{\otimes n} \xra{\bar \Delta_C^{\otimes n}} [(\bar C\otimes) k \oplus (\bar C \otimes \bar C) \oplus (k \otimes C)]^{\otimes n} \hra T(C) \otimes T(C)$$
landing in $\bar C^{\otimes l} \otimes \bar C^{\otimes m}$.  A nice way to describe this is as follows.

Let $\n$ be the ordered set with $n$ elements.  A surjection $\gamma: \bold p \ora \bold n$ induces a diagonal map $\Delta_{\gamma}: \bar C^{\otimes n} \ra \bar C^{\otimes p}$ by tensoring together the $n$ iterated diagonal maps $\bar C \ra \bar C^{\otimes \# p^{-1}(i)}$.  Then $\displaystyle \Delta^n_{l,m}= \sum_{\gamma} \Delta_{\gamma}$, with the sum running over surjections
$\gamma: \bold {l+m} \ora \bold n$ that are order preserving inclusions when restricted to either the first $l$ elements or the last $m$ elements.  Such a $\gamma$ corresponds to a pair of order preserving inclusions $\alpha: \bold l \ra \bold n$ and $\beta: \bold m \ra \bold n$ whose images cover all of $\bold n$.

\subsection{The cofree Hopf algebra $J^{\vee}(A)$}  We dualize the previous discussion.

The functor $J^{\vee}: \A^*(k) \ra \Ho^*(k)$ is defined as the right adjoint to the forgetful functor.  Then, as an coalgebra, $J^{\vee}(A)$ will be $\displaystyle T(A) = \bigoplus_{k=0}^{\infty} \bar A^{\otimes k}$, this time viewed as the cofree coalgebra cogenerated by $\bar A$.
The product $\nabla_{J^{\vee}(A)}$ on $J^{\vee}(C)$ is then induced by the product $\nabla_A$ on $A$: $\nabla_{J^{\vee}(A)}$ is the coalgebra map
$$ T(A) \otimes T(A) \ra T(A)$$
coinduced by the composite
$$ T(A) \otimes T(A) \twoheadrightarrow (\bar A\otimes k) \oplus (\bar A \otimes \bar A) \oplus (k \otimes A)= \overline{(A\otimes A)} \xra{\bar \nabla_A} \bar A.$$

The component $\nabla_{l,m}^n: \bar A^{\otimes l} \otimes \bar A^{\otimes m} \ra \bar A^{\otimes n}$ of $\nabla_{J^{\vee}(A)}$ has the following  description.
A surjection $\gamma: \bold p \ora \bold n$ defines a map
$$ \nabla_{\gamma}: \bar A^{\otimes p} \ra \bar A^{\otimes n}$$
by the formula
$$\nabla_{\gamma}(a_1 \otimes \cdots \otimes a_p) = (\prod_{j \in \gamma^{-1}(1)} a_j) \otimes \dots \otimes (\prod_{j \in \gamma^{-1}(n)} a_j).$$
Then $\displaystyle \nabla^n_{l,m}= \sum_{(\alpha, \beta)} \nabla_{\alpha+\beta}$, with the sum running over pairs of order preserving inclusions $\alpha: \bold l \ra \bold n$, $\beta: \bold m \ra \bold n$ whose images cover all of $\bold n$.

\begin{ex} Let $*$ be the product in $J^{\vee}(A)$.  Given $w,x,y,z \in \bar A$, all in even degrees if $k$ is not of characteristic 2,
\begin{equation*}
\begin{split}
(w\otimes x) \ast (y\otimes z) = & \text{\hspace{.05in}} [w \otimes x \otimes y \otimes z +  w \otimes y \otimes x \otimes z + w \otimes y \otimes z \otimes x
\\ & + y \otimes w \otimes x \otimes z + y \otimes w \otimes z \otimes x + y \otimes z \otimes w \otimes x]
\\ & +  [w \otimes x y \otimes z + w \otimes y \otimes x z+ w y \otimes x \otimes z
\\ & +  w y \otimes z \otimes x+ y \otimes w \otimes x z+ y \otimes w z \otimes x]
\\ & + w y \otimes x z.
\end{split}
\end{equation*}
Note that first group of terms here correspond to pairs  $\alpha, \beta: \bold 2 \ra \bold 4$ of order preserving inclusions whose images don't overlap: this is the classical shuffle product, and corresponds to $J^{\vee}(C)$ in the case when the product on $\bar C$ is zero.

\end{ex}

\section{Proofs of \thmref{char 0 theorem}, \thmref{char p theorem}, and \thmref{char p theorem dualized}} \label{easy proofs sec}

In this section we prove the parts of our theorems that are consequences of the classical theory \cite{milnor moore}.

\subsection{Char $k =0$: the proof of \thmref{char 0 theorem}} Let $k$ be a field of characteristic 0.  Recall that \thmref{char 0 theorem} went as follows.  Given $A, B \in \A^*(k)$, the following are equivalent. \\

\noindent{\bf (a)} $\bar A \simeq \bar B$, as graded $k$--vector spaces. \\

\noindent{\bf (b)} $J^{\vee}(A) \simeq J^{\vee}(B)$, as algebras. \\

\noindent{\bf (c)} $J^{\vee}(A) \simeq J^{\vee}(B)$, as Hopf algebras. \\

The implication (c)$\Rightarrow$(b) is clear.

The implication (b)$\Rightarrow$(a) follows from a Poincar\'e series argument.  Let $\chi_M(t)$ be the Poincar\'e series of a graded vector space $M \in \M(k)$.  Then
$$ \chi_{J^{\vee}(A)}(t) = \sum_{k=0}^\infty \chi_{\bar A}(t)^k = 1/(1-\chi_{\bar A}(t),$$
so that $\chi_{\bar A}(t) = 1-1/\chi_{J^{\vee}(A)}(t)$.
If (b) holds, then $\chi_{J^{\vee}(A)}(t) = \chi_{J^{\vee}(B)}(t)$, and so $\chi_{\bar A}(t) =\chi_{\bar B}(t)$.

Finally, the implication (a)$\Rightarrow$(c) will follow from the observation that there will be a Hopf algebra isomorphism $J^{\vee}(A) \simeq J^{\vee}(A_{triv})$, where $A_{triv}$ is $A$ given the trivial multiplication\footnote{In other words, $xy=0$ for all $x,y \in \bar A$.  Thus $J^{\vee}(A_{triv})$ is the graded shuffle algebra generated by $\bar A$.}.  Newman and Radford show this in \cite[Theorem 1.12]{newman radford}. An alternative proof goes as follows.  Since $J^{\vee}(A)$ is commutative, \cite[Proposition 4.17]{milnor moore} tells us that the composite $$\bar A = PJ^{\vee}(A) \xra{i} \overline{J^{\vee}(A)} \xra{\pi} QJ^{\vee}(A)$$ is monic.  Choose a linear map $q: QJ^{\vee}(A) \ra \bar A$ such that $q\circ \pi \circ i$ is the identity.  Now note that $q \circ \pi$ corresponds to an algebra map $q \circ \pi: J^{\vee}(A) \ra A_{triv}$, and this, in turn, corresponds to a map of Hopf algebras $\tilde q: J^{\vee}(A) \ra J^{\vee}(A_{triv})$.  By construction, the induced map on primitives, $P\tilde q: \bar A \ra \bar A_{triv}$, is the identity, and thus $\tilde q$ is an isomorphism, as $J^{\vee}(A)$ and $J^{\vee}(A_{triv})$ are cofree coalgebras.

\subsection{Char $k =p>0$: the proofs of \thmref{char p theorem}, and \thmref{char p theorem dualized}}

Now let $k$ be a perfect field of characteristic p.  Recall that \thmref{char p theorem} went as follows.  Given $A, B \in \A^*(k)$, the following are equivalent. \\

\noindent{\bf (a)} $\bar A \simeq \bar B$, as graded $F$--modules. \\

\noindent{\bf (b)} $J^{\vee}(A) \simeq J^{\vee}(B)$, as algebras. \\

\noindent{\bf (c)} $J^{\vee}(A) \simeq J^{\vee}(B)$, as Hopf algebras. \\

\thmref{char p theorem dualized} was the equivalent dual formulation.

The implication (c)$\Rightarrow$(b) is clear, and the hardest implication, (a)$\Rightarrow$(c) was proved (in the dual version) in the introduction, using \thmref{split thm}.

Putting these implications together shows that (a)$\Rightarrow$(b), but we note that there is a direct easier proof of this. Since $J^{\vee}(A) = TA = \bigoplus_{k=0}^{\infty}  \bar A^{\otimes k}$ as an $F$--module, $\bar A$ as an $F$--module clearly determines $\overline{J^{\vee}(A)}$ as an $F$--module.  So the next lemma shows that (a)$\Rightarrow$(b).

\begin{lem} \label{alg structure lem} The algebra structure of any $H \in \Ho^*(k)$ is determined by the $F$--module structure of $\bar H$.
\end{lem}
\begin{proof}  The point is that there is a structure theorem for algebras underlying a connected commutative Hopf algebra.  Define monogenic algebras as follows: if $p$ and $n$ are odd, let $A(n,0) = \Lambda(x)$, with $|x|=n$. If $p=2$ or $n$ is even, let
\begin{equation*}
A(n,j) =
\begin{cases}
k[x]/(x^{p^{j+1}}) & \text{if } 0 \leq j < \infty \\ k[x] & \text{if } j= \infty.
\end{cases}
\end{equation*}
where $|x|=n$.  Together, Proposition 7.8 and Theorem 7.11 of \cite{milnor moore} imply that, as an algebra, any $H \in \Ho^*(k)$ is the tensor product of $A(n,j)'s$.

So suppose $A \in \A^*(k)$ is the tensor product of such monogenic algebras.  We check that $A$ is determined by its structure as an $F$--module.  Let $A[n] \subset A$ be the subalgebra generated by elements in $A$ of degree at most $n$.  Then $A(0)=k$, and, for $n\geq 1$, if the sub $F$--module of $A\otimes_{A[n-1]} k$ generated by $n$--dimensional classes is isomorphic to $\displaystyle \bigoplus_i M(n,j_i)$, then
$$A[n] \simeq A[n-1] \otimes \bigotimes_i A(n,j_i).$$
Inductively, we see that each of the algebras $A[n]$ is determined by $F$--module structure, and thus the same is true for $A$.
\end{proof}

To finish the proof of \thmref{char p theorem}, it remains to check that (b)$\Rightarrow$(a). Since $J^{\vee}(A) = TA$ as an $F$--module, it suffices to prove the next lemma.

\begin{lem}  Let $N$ be an $F$--module.  Then $N$ is determined by the $F$--module $TN$.
\end{lem}

\begin{proof}  We prove this using a slightly elaborate Poincar\'e series argument.

Recall that \thmref{F module class thm} said that an $F$--module $N$ can be uniquely decomposed as a direct sum of the basic $F$--modules $N(n,j)$ defined in Definition \ref{F module defn 2}.  Thus we have a direct sum decomposition
 $$ N = \bigoplus_{n=1}^{\infty} \bigoplus_{j=0}^{\infty} a(n,j)N(n,j),$$
where $a(n,j) = 0$ if $p$ and $n$ are odd and $j >0$.

We let $a_N^j(t) = \sum_{n=1}^{\infty} a(n,j)t^n$. Our goal is to show that the set of power series $\{a_N^j(t), 1 \leq j \leq \infty\}$ is determined by $\{a_{TN}^j(t), 1 \leq j \leq \infty\}$.

We do this by switching to a more convenient set of power series.  It is useful to regard the $F$--module structure on $N$ as a nondegree preserving twisted linear map $F: N\ra N$.

For $1\leq j<\infty$, let $\chi_N^j(t) = \chi_{\coker\{F^{j+1}: N \ra N\}}(t)$, and let $\chi_N^{\infty}(t) = \chi_{N/tor(N)}(t)$, where $tor(N) = \{x \in N \ | \ F^jx=0 \text{ for }j>>0\}$.  It is not hard to see that when $1\leq j<\infty$, $$\chi_N^j(t) =  b_N^0(t) + b_N^1(t^p)+ \dots + b_N^j(t^{p^j})$$
where $\displaystyle b_N^j(t) = \sum_{k= j}^{\infty} a^k_N(t) + a^{\infty}_N(t)$, and also that $\displaystyle \chi_N^{\infty}(t) = \sum_{j=0}^{\infty} a^{\infty}_N(t^{p^j})$.

It follows that $a_N^{\infty}(t) = \chi_N^{\infty}(t)- \chi_N^{\infty}(t^p)$, and that
\begin{equation*}
\begin{split}
a^j_N(t^{p^{j+1}}) &
= b_N^j(t^{p^{j+1}}) - b_N^{j+1}(t^{p^{j+1}})  \\
  & = [\chi_N^j(t^p)-\chi_N^{j-1}(t^p)] - [\chi_N^{j+1}(t) - \chi_N^j(t)].
\end{split}
\end{equation*}
Thus the set $\{a_N^j(t), 1 \leq j \leq \infty\}$ determines and is determined by $\{\chi_N^j(t), 1 \leq j \leq \infty\}$.

Finally we check that $\chi_N^j(t)$ determines and is determined by $\chi_{TN}^j(t)$ for each $j$.  This follows by observing that $\chi_{N^{\otimes k}}^j(t) = (\chi_N^j(t))^k$, so that
$$ \chi_{TN}^j(t) = 1/(1-\chi_N^j(t)).$$
\end{proof}

\section{Examples, applications, and related results} \label{applications sec}

\subsection{A characterization of primitive Hopf algebras}

The `if' part of the next theorem is a nice application of \thmref{lifting thm}.

\begin{thm}  Let $k$ be a perfect field of characteristic $p>0$. A Hopf algebra $H \in \Ho_*(k)$ is primitively generated if and only if $H$ is split and $QH$ is a trivial $V$--module.
\end{thm}

We note that primitively generated $H \in H_*(k)$ are precisely the enveloping algebras of $p$--restricted Lie algebras by \cite[Theorem 6.11]{milnor moore}.

\begin{proof}  It is clear from the definitions that for all $H \in  \Ho_*(k)$, $PH$ is contained in the kernel of  $V: \bar H \ra \bar H$.  Thus if $H \in H_*(k)$ is primitively generated, i.e.~ the composite $PH \ra \bar H \ra QH$ is onto, then clearly  $QH$ is a trivial $V$--module split by any $k$--linear section of this composite.

Conversely, if $H$ is split, then \thmref{lifting thm} implies that there is a Hopf algebra morphism $G: H(QH) \ra H$ that induces the identity on indecomposables.  Such an $G$ will necessarily be onto. If $QH$ is also a trivial $V$--module, then $H(QH)$ is primitively generated by construction, and we deduce that $H$ will be primitively generated as well.
\end{proof}

\subsection{The Hopf algebra $H_*(\Omega \Sigma Y;k)$ is a stable invariant}

If $p$ is a prime and $X$ is a space, then the $p$th power map on $H^*(X;\F_p)$ identifies with a Steenrod operation:
\begin{equation*}
x^p =
\begin{cases}
Sq^{|x|}x & \text{if } p=2 \\ P^{|x|/2}x & \text{if $p$ is odd and $|x|$ is even}.
\end{cases}
\end{equation*}
As Steenrod operations are stable operations, we conclude that the $F$--module structure on $H^*(X;\F_p)$ is determined by the stable homotopy type of $X$.  Recalling that $J^{\vee}(H^*(X;k)) = H^*(\Omega \Sigma X;k)$ for $X$ of finite type, \thmref{char p theorem} (together with the easier \thmref{char 0 theorem}) has the following topological consequence.

\begin{cor}
If two based spaces $X$ and $Y$ of finite type are stably homotopy equivalent, then $H^*(\Omega \Sigma X;k)$ and $H^*(\Omega \Sigma Y;k)$ are isomorphic Hopf algebras for all fields $k$.
\end{cor}

\begin{rem} In the situation of the corollary, must there exist a Hopf algebra isomorphism  $H^*(\Omega \Sigma X;k) \simeq H^*(\Omega \Sigma Y;k)$ preserving all Steenrod operations?
\end{rem}

\subsection{When is a quasi-shuffle algebra polynomial?}

Let $k$ be a perfect field of characteristic $p$.  As was noted in \lemref{alg structure lem}, classical theory as in \cite{milnor moore} implies that the algebra structure of any $H \in \Ho^*(k)$ is determined by its structure as an $F$--module. In particular, one learns that $H$ will be a polynomial algebra (on even dimensional classes, if $p$ is odd) exactly when $F: \bar H \ra \bar H$ is monic.

Since $J^{\vee}(A) = \bigoplus_{k=0}^{\infty}  \bar A^{\otimes k}$ as an $F$--module, we see that $F: \overline{J^{\vee}(A)} \ra \overline{J^{\vee}(A)}$ will be monic if and only if $F: \bar A \ra \bar A$ is monic. We thus have the following theorem.

\begin{thm}  Let $k$ be a perfect field of characteristic $p$. Given a commutative algebra $A \in \A^*(k)$, the quasishuffle algebra $J^{\vee}(A)$ will be polynomial exactly in the following cases:
\begin{itemize}
\item $p=2$ and the squaring map $F: A \ra A$ is monic.
\item $p$ is odd, $A$ is concentrated in even dimensions, and the $p$th power map $F: A \ra A$ is monic.
\end{itemize}
\end{thm}

The algebra of quasi-symmetric functions on $k$ is the Hopf algebra
$$QSym(k) = J^{\vee}(k[y]) = H^*(\Omega \Sigma \mathbb C \mathbb P^{\infty};k),$$ where $|y|=2$.

\begin{cor} $QSym(k)$ is polynomial over every field.
\end{cor}

We briefly discuss how to extend our results to integral ones.  One lets $\A^*(\Z)$ and $\Ho^*(\Z)$ respectively be the categories of connected graded commutative algebras and Hopf algebras that are finitely generated free abelian groups in each degree.  Analogously to before, one then defines
$$ J^{\vee}: \A^*(\Z) \ra \Ho^*(\Z)$$
to be right adjoint to the forgetful functor.

For example, the universal ring of quasi-symmetric functions, $QSym(\Z)$, identifies with $J^{\vee}(\Z[y]) = H^*(\Omega \Sigma \mathbb C \mathbb P^{\infty};\Z)$, where $|y|=2$.

The so-called Ditters conjecture was the conjecture that $QSym(\Z)$ is polynomial.  This seems to have been first proved by Hazewinkel in \cite{hazewinkel}, who first notes that one can just check that $QSym(\F_p)$ is polynomial for all primes $p$.  Even better, Baker and Richter's method of making this reduction,   \cite[Proposition 2.4]{baker richter},  applies to show that given $A \in \A^*(\Z)$,  $J^{\vee}(A)$ will be polynomial if and only if $J^{\vee}(A) \otimes \F_p = J^{\vee}(A \otimes \F_p)$ is polynomial for all primes $p$.

We can thus conclude the following.

\begin{cor} Given  $A \in \A^*(\Z)$,  $J^{\vee}(A)$ will be polynomial if and only if for all primes $p$ and all $x \in A$, $x^p \equiv 0 \mod p \Rightarrow x \equiv 0 \mod p$.
\end{cor}

\begin{ex} Let $y$ and $z$ both have even grading.  Then $J^{\vee}(\Z[y,z]/(yz))$ is polynomial.
\end{ex}

\begin{rem} Hazewinkel doesn't use the coalgebra structure in his proof that $QSym(\F_p)$ is polynomial, but Baker and Richter \cite{baker richter} more sensibly do, with a proof like we have here.
\end{rem}

\subsection{An example: $H^*(\Omega \Sigma \mathbb C \mathbb P^2; k)$}

Let $k$ be a perfect field of characteristic $p$. Even when $F: \bar A \ra \bar A$ isn't monic, it isn't hard to use the $F$--module structure of $A \in \Ho^*(k)$ to explictly determine the algebra structure of $J^{\vee}(A)$.

We illustrate this by describing the algebra structure of $H^*(\Omega \Sigma \mathbb C \mathbb P^2; k) = J^{\vee}(k[y]/(y^3))$ for prime fields $k$.

First consider the case when $k=\Q$.  We know that $H^*(\Omega \Sigma \mathbb C \mathbb P^2; \Q)$ will be polynomial, with generators dual to the primitives of $H_*(\Omega \Sigma \mathbb C \mathbb P^2; \Q)$. These in turn correspond to a basis for the free Lie algebra generated by $\tilde H_*(\mathbb C \mathbb P^2; \Q)$, which is a two dimensional vector space with basis elements of homogeneous degree 2 and 4.

Now consider the case when $k=\F_p$ with $p$ an odd prime.  As an $F$--module, $\tilde H^*(\mathbb C \mathbb P^2;\F_p)$ is trivial, and thus the same is true for  $H^*(\Omega \Sigma \mathbb C \mathbb P^2; \F_p)$.  By dimension counting, one sees that, if
$$H^*(\Omega \Sigma \mathbb C \mathbb P^2; \Q) \simeq \bigotimes_i \Q[y_i],$$
then
$$H^*(\Omega \Sigma \mathbb C \mathbb P^2; \F_p) \simeq \bigotimes_i \bigotimes_{j=0}^{\infty} \F_p[y_{i,j}]/(y_{i,j}^p),$$
with $|y_{i,j}| = p^j|y_i|$.

Finally, consider the case when $k=\F_2$. Now $\tilde H^*(\mathbb C \mathbb P^2;\F_2) \simeq N(2,1)$, the $F$--module with basis $\{y,F(y)\}$, where $|y|=2$.  It is easy to check that
$$ N(2,1)^{\otimes k} = N(2k,1) \oplus \text{(a trivial $F$--module)},$$
and thus, as an $F$--module,
$$H^*(\Omega \Sigma \mathbb C \mathbb P^2; \F_2) \simeq \bigoplus_{k=1}^{\infty} N(2k,1) \oplus \text{(a big trivial $F$--module)}.$$
From this, one can conclude that, as an algebra
$$H^*(\Omega \Sigma \mathbb C \mathbb P^2; \F_2) \simeq \bigotimes_{j=1}^{\infty} \F_2[x_j]/(x_j^4) \otimes \text{(a big exterior algebra)},$$
where $|x_j| = 2^j$.

\end{document}